\documentclass[a4paper]{article}

\usepackage[pages=all, color=black, position={current page.south}, placement=bottom, scale=1, opacity=1, vshift=5mm]{background}
\usepackage[margin=1in]{geometry} % full-width

\usepackage{amsmath}
\usepackage{amsthm}
\usepackage{amssymb}
\usepackage{dsfont}
\usepackage{mathbbol}

\usepackage[utf8]{inputenc}
\usepackage{hyperref}
\hypersetup{
	unicode,
	pdfauthor={Tarek Emmrich},
	pdftitle={Chebotarëv's nonvanishing minors for eigenvectors of random matrices and graphs},
	pdfsubject={Nonvanishing minors},
	pdfkeywords={Nonvanishing minors, Galois groups},
	pdfproducer={LaTeX},
	pdfcreator={pdflatex}
}

\usepackage[sort&compress,numbers,square]{natbib}

\theoremstyle{plain}
\newtheorem{theorem}{Theorem}
\newtheorem{corollary}[theorem]{Corollary}
\newtheorem{lemma}[theorem]{Lemma}

\theoremstyle{definition}
\newtheorem{definition}[theorem]{Definition}
\newtheorem{example}[theorem]{Example}

\usepackage{graphicx, color}
\graphicspath{{fig/}}

\newcommand{\R}{\mathbb{R}}
\newcommand{\C}{\mathbb{C}}
\newcommand{\N}{\mathbb{N}}
\newcommand{\Z}{\mathbb{Z}}
\newcommand{\Q}{\mathbb{Q}}
\newcommand{\K}{\mathbb{K}}

\newcommand{\Gal}{\operatorname{Gal}}
\renewcommand{\l}{\lambda}

\usepackage{algorithm, algpseudocode} % use algorithm and algorithmicx for typesetting algorithms
\usepackage{mathrsfs} % for \mathscr command

\usepackage{lipsum}

\usepackage{tikz}

\title{Chebotarëv's nonvanishing minors for eigenvectors of random matrices and graphs}
\author{Tarek Emmrich}

\date{
}

\begin{document}

%\maketitle

%\include{UP}
\maketitle

\begin{abstract}
    For a matrix $\mathbf{M} \in \mathbb{K}^{n \times n}$ we establish a condition on the Galois group of the characteristic polynomial $\varphi_\mathbf{M}$ that induces nonvanishing of the minors of the eigenvector matrix of $\mathbf{M}$. For $\mathbb{K}=\Z$ recent results by Eberhard show that, conditionally on the extended Riemann hypothesis, this condition is satisfied with high probability\footnote{We say "with high probability" for probability $1-o(1)$ as $n \to \infty$. } and hence with high probability the minors of eigenvector matrices of random integer matrices are nonzero. For random graphs this yields a novel uncertainty principle, related to Chebotarëv's theorem on the roots of unity and results from Tao\cite{Tao} and Meshulam\cite{Meshulam}. We also show the application in graph signal processing and the connection to the rank of the walk matrix.
    
 	
    \end{abstract}

\section{Introduction}

For the Fourier matrix there is Chebotarëv's famous theorem on the nonvanishing of the minors, see \cite{Chebotarev} for a survey.
\begin{theorem}{\cite{Chebotarev}}
    Let $p$ be a prime and $\mathbf{F}$ be the Fourier matrix of order $p$, i.e., $\mathbf{F}_{i,j}=\omega^{(i-1)(j-1)}$ for $\omega=e^{\frac{2 \pi i}{p}}$. Then all minors of $\mathbf{F}$ are nonzero.
\end{theorem}
From the viewpoint of spectral graph theory, the matrix $\mathbf{F}$ is a possible eigenvector matrix of the circle graph and this yields the uncertainty principle
\[
\|f\|_0 + \| \hat{f}\|_0 \geq p+1
\]
for any function $f$ on the circle graph and its Graph Fourier transform $\hat{f}=\mathbf{F}^*f$, see Tao \cite{Tao} and Meshulam \cite{Meshulam} for uncertainty principles. We will prove the following criterion on nonvanishing of the minors of the eigenvector matrix of any matrix $\mathbf{M} \in \mathbb{K}^{n \times n}$. Let $\varphi_\mathbf{M}$ be the characteristic polynomial of $\mathbf{M}$ and $\mathbb{L}$ its splitting field. We denote by $\operatorname{Gal}(\varphi_\mathbf{M})$ the Galois group of the field extension $\mathbb{K} \subseteq \mathbb{L}$. Furthermore, if $\mathbf{M}$ is diagonalizable, let $\mathbf{U}$ be the matrix with columns $u_i$ that are eigenvectors corresponding to the eigenvalue $\lambda_i$ of $\mathbf{M}$.
\begin{theorem}\label{MainTheorem1}
    Let $\mathbf{U}$ be the eigenvector matrix of a diagonalizable matrix $\mathbf{M}$. If $\Gal(\varphi_\mathbf{M})\geq A_n$, then all minors of $\mathbf{U}$ are nonzero.
\end{theorem}
Note that for $\operatorname{Gal}(\varphi_\mathbf{M})\geq A_n$, $\mathbf{M}$ is always diagonalizable. For various reasons we care about the minors of the eigenvector matrices of operators an graphs. The two widely studied operators on graphs are the adjacency matrix and the Laplacian matrix. For directed graphs the directed adjacency matrix seems to be the most interesting.

%Let $\Q \subseteq K$ be a number fields, e.g. $K=\Q$, and $M \in K^{n \times n}$ a diagonizable matrix with $M=U \Lambda U^{-1}$. Let $\varphi$ be the characteristic polynomial of $M$ and $L$ the splitting field of $\varphi$. We denote by $\Gal(\varphi):=\Gal(L|K)$ the Galois group of the fields extension $K \subseteq L$. 

For a fixed finitely supported measure $\mu$ on $\Z$, let $\mathbf{M}$ be a random matrix with $m_{i,j} \sim \mu$. In this case Eberhard \cite{Eberhard} proved that, conditionally on the extended Riemann hypothesis, with high probability $\varphi_\mathbf{M}$ is irreducible and $\operatorname{Gal}(\varphi_\mathbf{M}) \geq A_n$.
%\begin{theorem}\label{EberhardThm}
%    Assume the extended Riemann hypothesis. Then with high probability $\varphi$ is irreducible and $\operatorname{Gal}(\varphi) \geq A_n$.
%\end{theorem}
This directly leads to the following corollary, which induces an uncertainty principle for directed graphs and their adjacency matrices.
\begin{corollary}\label{maincorollary}
    Assume the extended Riemann hypothesis. For a random matrix $\mathbf{M}$, with $m_{i,j} \sim \mu$, with high probability all minors of the eigenvector matrix $\mathbf{U}$ of $\mathbf{M}$ are nonzero.
\end{corollary}
Following Eberhard's results, Ferber, Jain, Sah and Sawhney \cite{FerberJain} could show that the Galois group of a random symmetric matrix is almost surely transitive, which we will see is not sufficient for minors of arbitrary size to be nonzero, but for minors of size one.
\begin{corollary}\label{corollaryadjacency}
    Let $\mathbf{M}$ be the adjacency matrix of a random Erdős–Rényi graph $G \sim G(n,p)$ and assume the extended Riemann hypothesis. Then with high probability no eigenvector of $\mathbf{M}$ contains a zero. 
\end{corollary}
Working with the Laplacian matrix of a random graph on $n$ vertices has the difficulty that the rank of the Laplacian matrix of a connected graph is $n-1$ by default. We can overcome this problem by using the factorization $\varphi_\mathbf{L}(\l)=\lambda \cdot \xi(\lambda)$, but still known results for random matrices do not apply. For applications, the most interesting case is the following uncertainty principle that follows directly from our results.
\begin{theorem}\label{GraphThm}
    Let $G$ be a graph on $n$ vertices and $\varphi_\mathbf{L}(\l)=\l\cdot \xi(\l)$. If $\operatorname{Gal}(\xi)\geq A_{n-1}$, then all minors of $\mathbf{U}$ are nonzero, i.e.
    \[
    \|f\|_0 + \| \hat{f}\|_0 \geq n+1
    \]
    for any signal $f \neq 0$ on the graph $G$. Here the Graph Fourier transform $\hat{f}$ is given by $\hat{f}=\mathbf{U}^{*}f$.
\end{theorem}
%Chebotarev proved the conjecture of Ostrovski that the minors of the fourier matrix of prime order,
%\[(w^{ij})_{{i=0,\ldots,p-1}\atop{j=0,\ldots,p-1}},\]
%do not vanish. In \cite{Tao}, Tao used this result to prove an uncertainty principle
%\[\|\|f\|_0 + \| \hat{f}\|_0 \geq p+1\]
%for the fourier transform induced by the group $\Z/p\Z$. Theorem \ref{GraphThm} is a similar result for the Graph fourier transform for any graph with $\Gal(P)\geq A_{n-1}$. One would expect that this is the case almost surely, but this has not been shown yet. 
The Fourier matrix can be seen as the eigenvector matrix of the shift matrix. Its characteristic polynomial is
\[
\lambda^p-1=(\lambda-1)\cdot \sum_{k=0}^{p-1}\lambda^k:=\lambda \cdot \xi(\lambda)
\]
with splitting field $\mathbb{L}=\Q[\lambda]/\xi(\lambda)$ and Galois group $(\Z/p\Z)^\times$, which is surprisingly not sufficient for our results to hold.

The eigenvectors of random matrices with continuous distribution have been widely studied in the past, for a survey see \cite{Survey}, but as far as we know the minors have not been studied in the discrete case.
In the discrete setting the most common application are graphs. Brooks and Lindenstrauss \cite{Lindenstrauss} proved a non-localization for the eigenvectors of the normalized Laplacian matrix of large regular graphs, i.e. small subsets of vertices only contain a small part of the mass of an eigenvector. Alt, Ducatez and Knowles \cite{JohannesAlt} determined regions for the edge probability $p=\frac{d}{N}$ of an Erdős–Rényi graph, such that the eigenvectors are delocalized. Both statements are relatively far away from a statement about minors. A very well studied object is the \textit{walk matrix} $\mathbf{W}$ of a graph and we will establish a connection between this rank and the minors of $\mathbf{U}$. The construction is well known in other areas of mathematics as Krylov spaces \cite{Saad}.

\textbf{Outline.} After introducing notations and the necessary background in Section 2, we will establish our main results and a slight variation of it in Section 3. We discuss several applications in Section 4. We end by extending the results to Laplacian matrices and by giving examples.

%In the appendix there will be a more technical version of the main theorem, this might for example be useful for more probabilistic results.

\section{Background}
In this section we will shortly introduce some notation and the background of Galois theory, for a more detailed summary see \cite{Morandi}. For a polynomial $\varphi \in \Q[\lambda]$ of degree $n$ the fundamental theorem of algebra tells us that $\varphi$ has exactly $n$ roots $\lambda_i$, $i=1,\ldots,n,$ over the complex numbers $\C$. Instead of working over the complex numbers, we will work over the \textit{splitting field} $\mathbb{L}=\Q(\lambda_i \colon i \in [n])$ of $\varphi$, which is the smallest field extension of $\Q$ that contains all roots of $\varphi$. The dimension of the splitting field is bounded by
\[
\operatorname{dim}_\Q(\mathbb{L})\leq n! \ .
\]
The splitting field can also be constructed inductively by the step $\Q \subseteq \mathbb{K}_1:=\Q[\lambda]/R(\lambda)$ that adds at least one root of $\varphi$, where $R(\lambda)$ is an irreducible factor of $\varphi$. After at most $n$ steps all roots of $\varphi$ are contained in the resulting field $\mathbb{L}$. The structure of this adding process is encoded in the Galois group \[\Gal(\varphi):=\operatorname{Aut}_{\Q}(\mathbb{L})\subseteq \operatorname{Sym}(\{\lambda_1,\ldots,\lambda_n\})\cong S_n.\]
%The same construction can also be applied to any field $K$ and any $\varphi \in K[\lambda]$, e.g. the Galois group of the field extension $\mathbb{F}_p \subseteq \mathbb{F}_{p^e}$ is the cylcic group of order $e$ generated by the Frobenius homomorphism $x \mapsto x^p$.
In most cases we will interpret the Galois group of any polynomial as a group acting on the set $[n]$ by the given isomorphism. In the generic case over the rational numbers, Hilbert's irreducibility theorem tells us that  \[\Gal(\varphi)\cong S_n.\]
Two algebraic numbers $\mu_1, \mu_2$ are called \textit{conjugate}, if they have the same minimal polynomial $\varphi$ over $\Q$. For two conjugate algebraic numbers $\mu_1, \mu_2$, there is a $g \in \Gal(\varphi)$ with $g(\mu_1)=\mu_2$. We call a group $H \subseteq \operatorname{Sym}(\{\l_1,\ldots,\l_n\})$ \textit{$m$-homogeneous}, if for all $S,S' \in \binom{[n]}{m}$ there exists a $g \in H$ such that
\[g(\{\lambda_i | i \in S\})=\{\lambda_i | i \in S'\}.\]
The prior fact tells us, that the Galois group of an irreducible polynomial is $1$-homogeneous (which in this case equals transitivity). For a matrix $\mathbf{U}$ and $W,S\subseteq [n]$ we write $\mathbf{U}_{W,S}$ for the submatrix of $\mathbf{U}$ with rows induced by $W$ and columns induced by $S$. 
%Under the assumption of Classification of simple groups, Cameron \cite{Cameron} proved already in 1981 that for $m\geq 6$ there is no $m$-transitive group other than $A_n$ and $S_n$.

For a tuple $B=(b_1,\ldots,b_k) \in \C^k$ we define the Vandermonde matrix
\[\mathbf{V}(B)=\left(b_j^i\right)_{{j=1,\ldots,k} \atop {i=0,\ldots,k-1}}.\]
It is well known that this matrix is invertible if and only if all values in $B$ are distinct.

For a finite set $V$ and $E\subseteq \binom{V}{2}$ we call $G=(V,E)$ a \textit{graph} and we often assume $V=[n]$. The elements of $V$ are called the \textit{vertices.} We say two vertices $v,w$ are connected, if $\{v,w\} \in E$ and sometimes write $v \sim w$. For two vertices $v,w$ the \textit{distance} $d(v,w)$ is the smallest integer $r$ such that there exist vertices $v=v_0, v_1, \ldots , v_r=w$ with $\{v_k,v_{k+1}\} \in E$.
The \textit{adjacency matrix} $\mathbf{A} \in \{0,1\}^{n \times n}$ of $G$ is defined by
\[
    \mathbf{A}_{v,w} = \begin{cases} 1, \text{ if } \{v,w\}\in E, \\ 0 , \text{ otherwise}.
    \end{cases}
\]
The \textit{combinatorial Laplacian matrix} $\mathbf{L} \in \Z^{n \times n}$ is the matrix $\mathbf{L}=\mathbf{D}-\mathbf{A}$, with $\mathbf{D}_{v,v}=\operatorname{degr}(v)$. For a connected graph, the rank of the Laplacian matrix is $n-1$ with the kernel vector $\mathds{1}$.

A matrix $\mathbf{U} \in \mathbb{K}^{n \times n}$ is orthonormal, if \[\mathbf{U}\mathbf{U}^*=\mathbf{U}^*\mathbf{U}=\mathbf{E}_n.\]
For an orthonormal matrix $\mathbf{U} \in \mathbb{K}^{n \times n}$ and a vector $f \in \K^n$ we denote
\[
\hat{f}:=\mathbf{U}^* f.
\]
If $\mathbf{U}$ is the eigenvector matrix of the Laplacian matrix, $\hat{f}$ is called the Graph Fourier transform, with the inverse transformation
\[
f=\mathbf{U}\hat{f}.
\]
We say a vector $f \in \mathbb{K}^n$ is $s$-sparse, if the size of its support equals $s$. We denote the size of the support of a vector $f$ also by
\[
\|f\|_0:=\#\{i \in [n] : f(i) \neq 0\}.
\]

\section{Minors and Galois groups}
Let $\mathbf{M} \in \mathbb{K}^{n \times n}$ be a diagonalizable matrix, let $\varphi_\mathbf{M}$ be its characteristic polynomial with splitting field $\mathbb{L}$. We denote the eigenvalues and eigenvectors by $\lambda_1,\ldots,\lambda_n \in \mathbb{L}$ and $u_1,\ldots,u_n \in \mathbb{L}^n$. The eigenvector matrix $\mathbf{U}$ is the matrix with columns $u_i$. For any two conjugate elements $\lambda_i, \lambda_j$ there exists $g \in \Gal (\varphi_\mathbf{M})$ such that $g(\lambda_i)=\lambda_j$. Applying this $g$ to the eigenvector equation $\mathbf{M}u_i=\lambda_i u_i$ we get 

\begin{align}\label{EVeq}
    g(\mathbf{M}) g(u_i) = \lambda_j g(u_i).
\end{align}

Since $\mathbf{M} \in \mathbb{K}^{n \times n}$ and $g_{|\mathbb{K}}$ is the identity, we have $g(\mathbf{M})=\mathbf{M}$ and hence the equation \eqref{EVeq} tells us $g(u_i)=u_j$. We denote by $u_i^J$ the restriction of the $i$-th eigenvector to the rows with indices in $J$. We are now ready to prove Theorem \ref{MainTheorem1}.
\begin{proof}[Proof of Theorem \ref{MainTheorem1}]
Assume there are sets $S,W \in \binom{[n]}{m}$ for some $m \in [n]$ such that $\det(\mathbf{U}_{W,S})=0$. Let $a_1,\ldots,a_m \in \mathbb{L}$ be the coefficients such that
\[
\sum_{i \in S}a_i u_i^W=0^W.
\]
Since $\Gal(\varphi_\mathbf{M})$ is $m$-homogeneous, there exists $g \in \Gal(\varphi_M)$ with $g(\{\l_i | i \in S\})=\{\l_i | i \in S'\}$ for all $S' \in \binom{[n]}{m}$. This yields
\[
\sum_{i \in S'}g(a_i)u_i^W=0^W
\]
for all $S' \in \binom{[n]}{m}$ and hence $\operatorname{rank}(\mathbf{U}_{W,[n]})=m-1$, a contradiction to the rows of $\mathbf{U}$ being linearly independent.
\end{proof}

Eberhard\cite[Theorem 1.3]{Eberhard} proved that, conditionally on the extended Riemann hypothesis, the Galois group of the characteristic polynomial of a random matrix is almost surely $A_n$ or $S_n$ and hence Corollary \ref{maincorollary} follows. Under the same condition Ferber, Jain, Sah and Sawhney \cite[Theorem 1.1]{FerberJain} could show that the characteristic polynomial of a random symmetrix matrix is almost surely irreducible. Thus Corollary \ref{corollaryadjacency} follows.

Our condition $\operatorname{Gal}(\varphi_\mathbf{M})\geq A_n$ is slightly stronger than just the $m$-homogeneity. Kantor\cite{Kantor}[Theorem 1] was able to show that only for $m=2,3,4$ there are groups that are $m$-homogeneous, but not $m$-transitive. Additionally, by the classification of finite simple groups, it was proven that for $m \geq 6$ there is no $m$-transitive group other than $A_n$ or $S_n$\cite{Cameron}. The most precise way to state the connection is the following.

\begin{theorem}
    Let $\mathbf{U}$ be the eigenvector matrix of a diagonalizable matrix $\mathbf{M}$ with characteristic polynomial $\varphi_\mathbf{M}$. If the Galois group $\Gal(\varphi_\mathbf{M})$ is m-homogeneous, then all minors of size $m$ and $(n-m)$ do not vanish. In particular, if the group $\Gal(\varphi_\mathbf{M})$ is transitive, i.e. the polynomial $\varphi_\mathbf{M}$ is irreducible, the matrix $\mathbf{U}$ does not contain a zero.
\end{theorem}

The proof relies on the proof of Theorem \ref{MainTheorem1} in connection with the following lemma. This lemma will also be useful for the eigenvector matrix of the Laplacian matrix, that requires some special treatment.

\begin{lemma}\label{complement}
    Let $\mathbf{U}$ be the a matrix with orthogonal columns $u_i$. If there are sets $W,S \in \binom{[n]}{m}$ such that the vectors $(u_i^W)_{i \in S}$ are linearly dependent, then the vectors $(u_i^{[n]\setminus W})_{i \in [n] \setminus S}$ are also linearly dependent.
\end{lemma}

\begin{proof}
    Let $a_i \in \R$ be the coefficients such that
    \[\sum_{i \in S}a_iu_i^W=0^W.\]
    For $j \notin S$ we have $\langle u_j , u_i\rangle=0$ for all $i \in S$ and it follows that
    \begin{align*}
    0&=\langle u_j,\sum_{i \in S}a_i u_i \rangle\\
    &=\langle u_j^W , \sum_{i \in S}a_i u_i^W \rangle + \langle u_j^{[n] \setminus W}, \sum_{i \in S} a_i u_i^{[n]\setminus W} \rangle \\
    &=0 + \langle u_j^{[n] \setminus W}, \sum_{i \in S} a_i u_i^{[n] \setminus W} \rangle.
    \end{align*}
    Since the vectors $u_i$ are linearly independent we know that $\sum_{i \in S}a_iu_i\neq 0^{[n]}$ and hence
    \[\sum_{i \in S} a_i u_i^{[n] \setminus W} \neq 0^{[n] \setminus W}.\]
    Now the $(n-m)$ vectors $u_j^{[n] \setminus W}$ of length $(n-m)$ lie in the hyperplane orthogonal to $\sum_{i \in S} a_i u_i^{[n] \setminus W}$ and therefore are linearly dependent.
\end{proof}

\begin{figure}[H]
\centering
\begin{tikzpicture}[scale=1]
    \filldraw[black] (0,0) circle (1.5pt) node[anchor=south] {3};
    \filldraw[black] (1,0) circle (1.5pt) node[anchor=south] {5};
    \filldraw[black] (2,0) circle (1.5pt) node[anchor=south] {7};
    \filldraw[black] (3,0) circle (1.5pt) node[anchor=south] {8};
    \filldraw[black] (0,-1) circle (1.5pt) node[anchor=north] {4};
    \filldraw[black] (1,-1) circle (1.5pt) node[anchor=north] {6};
    \filldraw[black] (-1,-1) circle (1.5pt) node[anchor=north] {2};
    \filldraw[black] (-1,0) circle (1.5pt) node[anchor=south] {1};
    \draw (0,0) -- (1,0);
    \draw (1,0) -- (2,0);
    \draw (2,0) -- (3,0);
    \draw (0,0) -- (0,-1);
    \draw (1,0) -- (1,-1);
    \draw (0,-1) -- (1,-1);
    \draw (0,-1) -- (-1,-1);
    \draw (0,0) -- (-1,0);
    \draw (-1,-1) -- (-1,0);
    \draw (0,-1) -- (-1,0);
\end{tikzpicture}
\caption{A graph with 1-homogeneous Galois group $\Gal (\varphi_\mathbf{A})<A_{8}$, but vanishing minors.}
\label{imprimitivezero}
\end{figure}
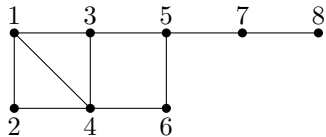
The characteristic polynomial of the adjacency matrix of the graph in figure \ref{imprimitivezero} has the characteristic polynomial

\[
\varphi_\mathbf{A}=\l^8-10\l^6-4\l^5+23\l^4+12*\l^3-12\l^2-4\l+1,
\]
that is irreducible, but with the Galois group

\[
\Gal(\varphi_\mathbf{A}) \cong \left\{\sigma \in S_8 \colon \sigma(\{1,2,3,4\}) \in \{\{1,2,3,4\},\{5,6,7,8\}\}\right\}.
\]

The problem of vanishing minors originally appeared in  the analysis of an algorithm \cite{SubspaceEJK}, that recovers sparse signals on graphs, whose success relies on the regularity of the matrix
\[
\mathbf{G}_{v,j}=f_v(j)
\]
with $f_v(j)=\sum_{i \in S}\lambda_i^ju_i(v)$ for $0 \leq j \leq s-1$ and $v \in W \subseteq [n]$. The factorization $\mathbf{G}=\mathbf{U}_{W,S}\cdot \mathbf{V}(\lambda_i : i \in S)$ yields the connection to minors. A slighty improved version of this algorithm relies on the regularity of the following matrix
\[
\tilde{\mathbf{G}}=
\begin{pmatrix}
        f_1(0) & f_1(1) & \ldots & f_1(s-1) \\
        f_1(1) & f_1(2) & \ldots & f_1(s) \\
        \vdots & \vdots & & \vdots \\
        f_1(r_1-1) & f_1(r_1) & \ldots & f_1(s-2+r_1) \\
        f_2(0) & f_2(1) & \ldots & f_2(s-1) \\
        f_2(1) & f_2(2) & \ldots & f_2(s) \\
        \vdots & \vdots & & \vdots \\
        f_k(r_k-1) & f_k(r_k) & \ldots & f_k(s-2+r_k) 
    \end{pmatrix}
\]
for some radii $\sum_{i=1}^kr_i=m$ and corresponding vertices $W=\{v_1,\ldots,v_k\}$. This matrix has the factorization
\[
\tilde{\mathbf{G}}=
\underbrace{
\begin{pmatrix}
    u_{j_1}(v_1)&u_{j_2}(v_1)&\ldots&u_{j_s}(v_1)\\
    \l_{j_1}u_{j_1}(v_1)&\l_{j_2}(v_1)&\ldots&l_{j_s}u_{j_s}(v_1)\\
    \vdots&\vdots& & \vdots\\
    \l_{j_1}^{r_1-1}u_{j_1}(v_2)&\l_{j_2}^{r_1-1}u_{j_2}(v_2)&\ldots&\l_{j_s}^{r_1-1}u_{j_s}(v_2)\\
    u_{j_1}(v_2)&u_{j_2}(v_2)&\ldots&u_{j_s}(v_2)\\
    \l_{j_1}u_{j_1}(v_2)&\l_{j_2}(v_2)&\ldots&l_{j_s}u_{j_s}(v_2)\\
    \vdots&\vdots& & \vdots\\
    \l_{j_1}^{r_k-1}u_{j_1}(v_k)&\l_{j_2}^{r_k-1}u_{j_2}(v_k)&\ldots&\l_{j_s}^{r_k-1}u_{j_s}(v_k)
\end{pmatrix}}_{:=\mathbf{G}(S,W,r_1,\ldots,r_k)}
\cdot
\mathbf{V}(\lambda_i : i \in S),
\]
for $W=\{v_i:1\leq i \leq k\}$ and $S=\{j_i : 1 \leq i \leq s\}$. Note that the rows of each vertex-block are linearly independent as soon as $u_j(v)\neq0$ for $j \in S$. The next matrix is a small example of such a matrix.
\[\mathbf{G}\left(\{1,2\},\{v,w\},2,2\right)=
\begin{pmatrix}
u_1(v)&u_2(v)&u_3(v)&u_4(v)\\
\lambda_1u_1(v)&\lambda_2u_2(v)&\lambda_3u_3(v)&\lambda_4u_4(v)\\
u_1(w)&u_2(w)&u_3(w)&u_4(w)\\
\lambda_1u_1(w)&\lambda_2u_2(w)&\lambda_3u_3(w)&\lambda_4u_4(w)
\end{pmatrix}.
\]
\begin{theorem}\label{radii}
    Let $\mathbf{U}$ be the eigenvector matrix of $\mathbf{A}$. The matrix $\mathbf{G}(S,W,r_1,\ldots,r_k)$ has full rank, if $\Gal(\varphi_\mathbf{A})\geq A_n$ and  $d(v_i,v_j)>r_i+r_j$ for all $1 \leq i,j \leq k$.
\end{theorem}

\begin{proof}
    Assume $\operatorname{rank}(\mathbf{G}(S,W,r_1,\ldots,r_k))<s$. Then there are coefficients $a_i$ such that
    \[
        \sum_{i \in S}a_i\l_i^{k_i}u_i(v)=0
    \]
    for $v \in W$ and $0 \leq k_i \leq r_k-1$. Applying each $\gamma \in \operatorname{Gal}(\varphi_A)$ yields $\operatorname{rank}(\mathbf{G}([n],W,r_1,\ldots,r_k))<s$. Let $Z_{v,m}$ be the row of $\mathbf{G}([n],W,r_1,\ldots,r_k)$ belonging to vertex $v$ and power $m$. We have
    \begin{align*}
        \langle Z_{v,m},Z_{w,l} \rangle &= \sum_{i=1}^n\l_i^{m+l}u_i(v)u_i(w)\\
        &=\mathbf{A}_{v,w}^{m+l}.
    \end{align*}
    And hence we know that $\langle Z_{v,m},Z_{w,l} \rangle=0$ for $d(v,w)>m+l$. This yields that the intersection of the span of the blocks is $0$, if the distance condition is fulfilled. Again, if $\Gal(\varphi_\mathbf{A})\geq A_n$, then the blocks have full rank and hence the matrix $\mathbf{G}([n],W,r_1,\ldots,r_k)$ can not be singular, a contradiction.
\end{proof}

\section{Consequences of nonvanishing minors and applications}
There are several consequences of the minors of the eigenvector matrix being nonzero. The first one is an uncertainty principle similar to the one from Tao.
\begin{theorem}\label{minoruncertainty}{\cite[Theorem~1.1]{Tao}}
    Let $\mathbf{U} \in \mathbb{K}^{n \times n}$ be a matrix. All minors of any size of $\mathbf{U}$ are nonzero if and only if \[\|f\|_0+\|\hat{f}\|_0\geq n+1.\]
\end{theorem}

\begin{proof}
    Assume some minor of $\mathbf{U}$ is zero. Then there are coefficients $a_i \in \mathbb{K}$ such that $\sum_{i \in I}a_iu_i^J=0^J$ for $\#I=\#J=m$. The vector
    \[
    f(i)=\begin{cases}a_i\ \text{ for } i \in I,\\
    0\ \ \text{  else.}
    \end{cases}
    \]
    has support of size $m$. This choice yields $(\mathbf{U}f)^J=0^J$ and hence
    \[\|f\|_0+\|\mathbf{U}f\|_0\leq m + (n-m)=n.
    \]
    Assume $\|f\|_0+\|\hat{f}\|_0=n$ and denote $\|f\|_0=r$. Let $W=\operatorname{supp}(f)$ be the support of $f$ and $S=\operatorname{supp}(\hat{f})$ be the support of $\hat{f}$ with $\#S=n-r$. Restricting $\mathbf{U}$ to the columns belonging to $S$ yields
    \[
    \mathbf{U}_{[n],S} \hat{f}_S=f,
    \]
    which means that the minor $\mathbf{U}_{[n]\setminus W,S}$ is zero, a contradiction.
\end{proof}

Another well known consequence is the following application to sparse recovery.
\begin{theorem}{\cite[Theorem~2.13]{FoRa13}}
    Let $G$ be a graph with vertex set $[n]$ and $\mathbf{M}=\mathbf{U}\Lambda \mathbf{U}^*$ for $\mathbf{M} \in \{\mathbf{L},\mathbf{A}\}$ with nonvanishing minors of $\mathbf{U}$ of any size. For any $1 \leq s \leq n$ pick any $W,S \in \binom{[n]}{s}$, then each $s$-sparse signal
    \[f=\sum_{w \in W}a_w\delta_w\]
    is uniquely determined by at least $2s$ samples of the form \[\langle u_i, f \rangle.\]
    %Conversely, each signal \[f=\sum_{i \in S}a_iu_i\] is uniquely determined by at least $2s$ samples of the form $f(v)$.
\end{theorem}

Let $\mathbf{U}$ be the eigenvector matrix of $\mathbf{M}  \in \{\mathbf{L},\mathbf{A}\}$ and let $0\neq x \in \R^n$ be a vector. One interesting object to study is the so called walk matrix, see e.g. \cite{Cvetkovi}\cite{Godsil}\cite{Liu}\cite{Rowlinson},
\[
\mathbf{W}(x):=
\begin{pmatrix}
    \vert & \vert & \vert & \vert & \vert\\
    x & \mathbf{M} x & \mathbf{M}^2 x & \cdots & \mathbf{M}^{n-1}x\\
    \vert & \vert & \vert & \vert & \vert
\end{pmatrix}
\]
and its rank, depending on $x$. Let $0\leq r \leq n-1$ be the minimal number such that there are coefficients $a_0,\ldots,a_r$ with
\[\sum_{k=0}^ra_k\mathbf{M}^kx=0^{[n]}.\]
Then $r$ is the degree of the polynomial $Q(\l)=\sum_{k=0}^ra_k\l^k$ and also the rank of $\mathbf{W}$, because each set of columns with indices in $\{i,\ldots,i+r\}$ will have the kernel \[(a_0,\ldots,a_r)^T.\]
The next theorem describes the connection between the choice of $x$, the rank of the walk matrix and the nonvanishing of minors of $\mathbf{U}$.
\begin{theorem}\label{WalkAdjacency}
    Let $G$ be a graph on $n$ vertices, $0 \neq x\in \R^n$ and let $\mathbf{W}(x)$ be the walk matrix for $\mathbf{M}=\mathbf{A}$.
    \begin{enumerate}
        \item If $\varphi_\mathbf{A}$ is irreducible and $x \in \Q^n$, then
        \[\operatorname{rank}(\mathbf{W}(x))=n.\]
        \item For more general $0 \neq x \in \R^n$ we have
        \[\operatorname{rank}(\mathbf{W}(x))=\|\mathbf{U}^*x\|_0=\|\hat{x}\|_0.\]
        \item If all minors of $\mathbf{U}$ are nonvanishing, the following formula holds
        \[\operatorname{rank}(\mathbf{W}(x))\geq n+1-\|x\|_0.\]
    \end{enumerate}
\end{theorem}

\begin{proof}
In a first step we will show the following equation \[\operatorname{rank}(\mathbf{W}(x))=n-\|\mathbf{U}^*x\|_0.\]
Let $Q=\sum_{k=0}^ra_k\l^k$ be the polynomial with coefficients $a_k$, that are the kernel of $\mathbf{W}(x)$, i.e.,
\begin{align*}
Q(\mathbf{A})\cdot x=\sum_{k=0}^ra_k\mathbf{A}^kx=0^{[n]},
\end{align*}
which is equivalent to
\begin{align}\label{firstcondition}
\sum_{k=0}^ra_k \sum_{w=1}^n\mathbf{A}_{v,w}^kx_w=0 \text{ for all }v.
\end{align}
The diagonalization of $\mathbf{A}$ yields
\[ \mathbf{A}_{v,w}^k=\sum_{l=1}^nu_l(v)u_l(w)\lambda_l^k\]
and we want to use this equation to obtain a formula for $\left(Q(\mathbf{A})\cdot x\right)$ that gives more insight
\begin{align*}
    \sum_{k=0}^ra_k \sum_{w=1}^n\mathbf{A}_{v,w}^kx_w &= \sum_{k=0}^ra_k\sum_{w=1}^n\left(\sum_{l=1}^n u_l(v)u_l(w)\lambda_l^k\right)x_w\\
    &=\sum_{l=1}^nu_l(v)\sum_{w=1}^nx_wu_l(w)\sum_{k=0}^ra_k\lambda_l^k\\
    &=\sum_{l=1}^nu_l(v)\sum_{w=1}^nx_wu_l(w)Q(\l_l).
\end{align*}
We can see now that equation \eqref{firstcondition} is equivalent to
\[\sum_{l=1}^nu_l(v)Q(\lambda_l)\sum_{w=1}^nx_wu_l(w)=0 \text{ for all }v.\]
This condition reads as
\[\left\langle \left(u_l(v)\right)_{l=1,\ldots,n},\left(Q(\lambda_l)\sum_{w=1}^nx_wu_l(w)\right)_{l=1,\ldots,n}\right\rangle =0 \text{ for all }v\]
and since the vectors $(u_l(v))_{l=1,\ldots,n}$ form a basis 
this yields \[Q(\lambda_l)\sum_{w=1}^nx_wu_l(w)=0\] for all $l$. The second vector can be written as
 \[\left(\sum_{w=1}^nw_xu_l(w)\right)_{l=1,\ldots,n}=\mathbf{U}^*x\]
Thus we have
\[
\operatorname{rank}(\mathbf{W})(x)=n-\|\mathbf{U}^*x\|_0.
\]

Now for the first point, observe that $x \in \Q^n$ implies $a_k \in \Q$ and hence $a_k \in \Z$ and thus $Q$ divides  $\varphi_\mathbf{A}$ in the polynomial ring $\Q[X]$. If $\varphi_\mathbf{A}$ is irreducible, either $Q=\varphi_\mathbf{A}$ or $Q=1$. The case $Q=1$ is not possible for $x \neq 0$ and hence $\operatorname{rank}(\mathbf{W})=n$ for $\varphi_\mathbf{A}$ irreducible and $x \in \Q^n$.

For not necessarily rational $x \in \R^n$ we see that 
\[
\operatorname{rank}(\mathbf{W}(x))=\|\mathbf{U}^*x\|_0=\|\hat{x}\|_0.
\]
If we restrict the size of the support of $x$, i.e., $\|x\|_0=s$, nonvanishing minors and Theorem \ref{minoruncertainty} yield
\[
\operatorname{rank}(\mathbf{W}(x))\geq n+1-s.
\]
\end{proof}

\section{Laplacian matrices}

We now formulate the analog versions of the prior theorems for the Laplacian matrices with necessary changes. Let $\xi(\l)$ be the unique polynomial with $\varphi_L=\l \cdot \xi(\l)$.

\begin{theorem}\label{maintheorem}
    Let $G$ be a graph on $n$ vertices and $\mathbf{U}$ be the eigenvector matrix of the Laplacian matrix $\mathbf{L}$ with characteristic polynomial $\varphi_L=\l \cdot \xi$.  If $\operatorname{Gal}(\xi)\geq A_{n-1}$, then we have
    \[
    \operatorname{det}\left( \mathbf{U}_{W,S}  \right) \neq 0
    \]
    for all $W,S \in \binom{[n]}{m}$.
\end{theorem}
\begin{proof}
    By Lemma \ref{complement} we can assume $1 \notin S$. If $\operatorname{det}\left(\mathbf{U}_{W,S}\right)=0$ for some $W,S \in \binom{[n]}{m}$, we know that $\operatorname{det}\left(\mathbf{U}_{W,S'}\right)$=0 for all $S' \in \binom{\{2,\ldots,n\}}{m}$ by applying the $m$-homogeneous Galois group to the equation and hence
    \[
    \operatorname{rank}\left(\mathbf{U}_{W,\{2,\ldots,n\}}\right)=m-1.
    \]
    Thus there are coefficients $a_w \in \K$ for $w \in W$ such that
    \[\sum_{w \in W} a_w \tilde{u}(w)=0,\]
    where $\tilde{u}(w)=(u_i(w))_{i=2,\ldots,n}$. For $v \notin W$ we know that $\langle u(v),u(w) \rangle=0$ and hence
    \begin{align*}
    0&=\langle u(v), \sum_{w \in W} a_w u(w) \rangle \\
    &= \langle u_1(v), \sum_{w \in W} a_w u_1(w) \rangle + \langle \tilde{u}(v),\sum_{w \in W} a_w \tilde{u}(v) \rangle \\
    &= u_1(v) \cdot \sum_{w \in W}a_wu_1(w).
    \end{align*}
    Now since the vectors $u(w)$ are linearly independent for $w \in W$ we know that $\sum_{w \in W}a_wu_1(w)\neq 0$ and together with $u_1(v)=\frac{1}{\sqrt{n}}\neq 0$ this yields a contradiction.
\end{proof}
Theorem \ref{GraphThm} follows now from the fact that $A_{n-1}$ and $S_{n-1}$ are $m$-homogeneous and Lemma \ref{complement}. Moreover, if $\xi$ is irreducible, it also follows directly from Theorem \ref{maintheorem} that all eigenvectors do not contain a zero. Furthermore, if $\operatorname{Gal}(\xi)$ is $m$-homogeneous, then all minors of $\mathbf{U}$ of size $m$ are nonzero.
\begin{theorem}
    Let $G$ be a graph on $n$ vertices, $0 \neq x\in \R^n$ and let $\mathbf{W}(x)$ be the walk matrix for $\mathbf{M}=\mathbf{L}$ with characteristic polynomial $\varphi_L=\l \cdot \xi$.
    \begin{enumerate}
        \item If $\xi(\l)$ is irreducible and $x \in \Q^n$, then
        \[\operatorname{rank}(\mathbf{W}(x))=n.\]
        \item For more general $0 \neq x \in \R^n$ we have
        \[\operatorname{rank}(\mathbf{W})=\|\mathbf{U}^*x\|_0=\|x\|_0.\]
        \item If all minors of $\mathbf{U}$ are nonvanishing, the following formula holds
        \[\operatorname{rank}(\mathbf{W})\geq n+1-\|x\|_0.\]
    \end{enumerate}
\end{theorem}
The only change in the proof in comparison to Theorem \ref{WalkAdjacency} is the use of $\l_1=0$ and $u_1=\mathds{1}$.
\begin{theorem}
    Let $G$ be a graph on $n$ vertices and $\mathbf{U}$ be the eigenvector matrix of the Laplacian matrix $\mathbf{L}$ with characteristic polynomial $\varphi=\l \cdot \xi$. The matrix $\mathbf{G}(S,W,r_1,\ldots,r_k)$ has full rank, if $\Gal(\xi)\geq A_{n-1}$ and  $d(v_i,v_j)>r_i+r_j$ for all $1 \leq i,j \leq k$.
\end{theorem}

\begin{proof}
    By a similar argument as in the proof of Theorem \ref{radii}, we get that singularity of $\mathbf{G}(S,W,r_1,\ldots,r_k)$ yields singularity of $\mathbf{G}(\{2,\ldots,n\},W,r_1,\ldots,r_k)$. Let $Z_{v,m}$ be the row of $\mathbf{G}(\{2,\ldots,n\},W,r_1,\ldots,r_k)$ belonging to vertex $v$ and power $m$. For $m+l>0$ we have
    \begin{align*}
    \langle Z_{v,m},Z_{w,l} \rangle &= \sum_{i=2}^n\l^{m+l}u_i(v)u_i(w)\\
    &=\mathbf{L}_{v,w}^{m+l}.
    \end{align*}
    The radius condition yields orthogonality for $m+l>0$. The rows belonging to power $0$ can not be linearly dependent according to Theorem \ref{maintheorem}.
\end{proof}

We finally provide two examples.
\begin{example}
The graph in Figure \ref{imprimitivenonzero} has the characteristic polynomial
\[\varphi=\lambda \cdot \left(\lambda^8 - 28\lambda^7 + 332\lambda^6 - 2170\lambda^5 + 8516\lambda^4 - 20440\lambda^3 + 29105\lambda^2 - 22288\lambda + 6957\right).\]
The Galois group  of $\xi$ has order $384$ and is given by
\[
\Gal (\xi) \cong \left\{ \sigma \in S_8 \colon \sigma(i+4)=\begin{cases} \sigma(i)+4 \text{ , for $\sigma(1) \leq 4$}, \\ \sigma(i)-4 \text{ , for $\sigma(1) \geq 5$,}  \end{cases} \text{ for } 1 \leq i \leq 4 \right\}.
\]
This group is imprimitive but the minors of any size are nonzero. We see that our criterion is not necessary.

\begin{figure}[H]
\centering
\begin{tikzpicture}[scale=1.5]
    \filldraw[black] (0,0) circle (1pt) node[anchor=south] {1};
    \filldraw[black] (1,0) circle (1pt) node[anchor=south] {2};
    \filldraw[black] (2,0) circle (1pt) node[anchor=south] {3};
    \filldraw[black] (0.5,-0.5) circle (1pt) node[anchor=south] {4};
    \filldraw[black] (1,-0.5) circle (1pt) node[anchor=west] {5};
    \filldraw[black] (1.5,-0.5) circle (1pt) node[anchor=north west] {6};
    \filldraw[black] (0,-1) circle (1pt) node[anchor=east] {7};
    \filldraw[black] (1,-1) circle (1pt) node[anchor=north west] {8};
    \filldraw[black] (0.5,-1.5) circle (1pt) node[anchor=north] {9};
    \draw (0,0) -- (1,0);
    \draw (1,0) -- (2,0);
    \draw (0,0) -- (0.5,-0.5);
    \draw (0,0) -- (0,-1);
    \draw (0.5,-0.5) -- (1,-0.5);
    \draw (1,0) -- (1,-1);
    \draw (0,-1) -- (0.5,-0.5);
    \draw (0,-1) -- (1,-1);
    \draw (1,-1) -- (1.5,-0.5);
    \draw (1.5,-0.5) -- (2,0);
    \draw (0,-1) -- (0.5,-1.5);
    \draw (0.5,-1.5) -- (1,-1);
    \draw (1,0) -- (1.5,-0.5);
\end{tikzpicture}
\caption{A graph with imprimitive Galois group $\Gal (\xi)<A_{8}$, but nonvanishing minors.}
\label{imprimitivenonzero}
\end{figure}
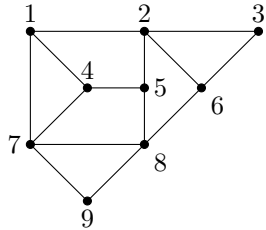
\end{example}

\begin{example}

The graph in Figure \ref{irreduciblezero} has the characteristic polynomial
\[\varphi=\lambda \cdot \left( \lambda^6-12\lambda^5+54\lambda^4-114\lambda^3+115\lambda^2-50\lambda+7\right).\]
The Galois group of $\xi$ has order 72 and is given by
\[
\Gal(\xi) \cong \left\{ \sigma \in S_6 \colon \sigma(\{1,2,3\})\in \{\{1,2,3\},\{4,5,6\}\}\right\}.
\]
This group is imprimitive and the vanishing minors of size 3 are supported on the columns $\{1,2,3\}$ or $\{4,5,6\}$.
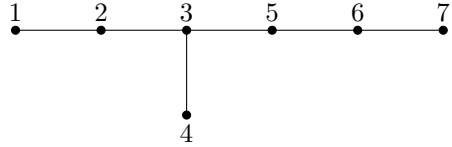
\begin{figure}[H]
\centering
\begin{tikzpicture}[scale=1.5]
    \filldraw[black] (0,0) circle (1pt) node[anchor=south] {1};
    \filldraw[black] (0.75,0) circle (1pt) node[anchor=south] {2};
    \filldraw[black] (1.5,0) circle (1pt) node[anchor=south] {3};
    \filldraw[black] (2.25,0) circle (1pt) node[anchor=south] {5};
    \filldraw[black] (3,0) circle (1pt) node[anchor=south] {6};
    \filldraw[black] (3.75,0) circle (1pt) node[anchor=south] {7};
    \filldraw[black] (1.5,-0.75) circle (1pt) node[anchor=north] {4};
    \draw (0,0) -- (0.75,0);
    \draw (0.75,0) -- (1.5,0);
    \draw (1.5,0) -- (2.25,0);
    \draw (2.25,0) -- (3,0);
    \draw (3,0) -- (3.75,0);
    \draw (1.5,0) -- (1.5,-0.75);
    
\end{tikzpicture}
\caption{A graph with $1$-homogeneous Galois group, but vanishing minors.}
\label{irreduciblezero}
\end{figure}
\end{example}

\bibliographystyle{abbrv}
\bibliography{refs}

\begin{thebibliography}{10}

\bibitem{JohannesAlt}
J.~Alt, R.~Ducatez, and A.~Knowles.
\newblock The completely delocalized region of the {E}rdős–{R}ényi graph.
\newblock {\em Electron. Commun. Probab.}, 27:Paper No. 10, 9, 2022.

\bibitem{Lindenstrauss}
S.~Brooks and E.~Lindenstrauss.
\newblock Non-localization of eigenfunctions on large regular graphs.
\newblock {\em Israel J. Math.}, 193(1):1--14, 2013.

\bibitem{Cameron}
P.~J. Cameron.
\newblock Finite permutation groups and finite simple groups.
\newblock {\em Bull. London Math. Soc.}, 13(1):1--22, 1981.

\bibitem{Cvetkovi}
D.~M. Cvetkovi\'{c}.
\newblock The main part of the spectrum, divisors and switching of graphs.
\newblock {\em Publ. Inst. Math. (Beograd) (N.S.)}, 23(37):31--38, 1978.

\bibitem{Eberhard}
S.~Eberhard.
\newblock The characteristic polynomial of a random matrix.
\newblock {\em Combinatorica}, 42(4):491--527, 2022.

\bibitem{SubspaceEJK}
T.~Emmrich, M.~Juhnke-Kubitzke, and S.~Kunis.
\newblock Two subspace methods for frequency sparse graph signals, 2023.

\bibitem{FerberJain}
A.~Ferber, V.~Jain, A.~Sah, and M.~Sawhney.
\newblock Random symmetric matrices: rank distribution and irreducibility of the characteristic polynomial.
\newblock {\em Math. Proc. Cambridge Philos. Soc.}, 174(2):233--246, 2023.

\bibitem{FoRa13}
S.~Foucart and H.~Rauhut.
\newblock {\em A mathematical introduction to compressive sensing}.
\newblock Applied and Numerical Harmonic Analysis. Birkh\"{a}user/Springer, New York, 2013.

\bibitem{Godsil}
C.~Godsil.
\newblock Controllable subsets in graphs.
\newblock {\em Ann. Comb.}, 16(4):733--744, 2012.

\bibitem{Kantor}
W.~M. Kantor.
\newblock {$k$}-homogeneous groups.
\newblock {\em Math. Z.}, 124:261--265, 1972.

\bibitem{Liu}
F.~Liu and J.~Siemons.
\newblock Unlocking the walk matrix of a graph.
\newblock {\em J. Algebraic Combin.}, 55(3):663--690, 2022.

\bibitem{Meshulam}
R.~Meshulam.
\newblock An uncertainty inequality for finite abelian groups.
\newblock {\em European J. Combin.}, 27(1):63--67, 2006.

\bibitem{Morandi}
P.~Morandi.
\newblock {\em Field and {G}alois theory}, volume 167 of {\em Graduate Texts in Mathematics}.
\newblock Springer-Verlag, New York, 1996.

\bibitem{Survey}
S.~O'Rourke, V.~Vu, and K.~Wang.
\newblock Eigenvectors of random matrices: a survey.
\newblock {\em J. Combin. Theory Ser. A}, 144:361--442, 2016.

\bibitem{Rowlinson}
P.~Rowlinson.
\newblock The main eigenvalues of a graph: a survey.
\newblock {\em Appl. Anal. Discrete Math.}, 1:445--471, 2007.

\bibitem{Saad}
Y.~Saad.
\newblock {\em Iterative methods for sparse linear systems}.
\newblock Society for Industrial and Applied Mathematics, Philadelphia, PA, second edition, 2003.

\bibitem{Chebotarev}
P.~Stevenhagen and H.~W. Lenstra, Jr.
\newblock Chebotar\"{e}v and his density theorem.
\newblock {\em Math. Intelligencer}, 18(2):26--37, 1996.

\bibitem{Tao}
T.~Tao.
\newblock An uncertainty principle for cyclic groups of prime order.
\newblock {\em Math. Res. Lett.}, 12(1):121--127, 2005.

\end{thebibliography}

\end{document}